\documentclass[preprint,11pt]{elsarticle}

\usepackage{amsfonts, amsmath, amscd}
\usepackage[psamsfonts]{amssymb}

\usepackage{amssymb}

\usepackage{pb-diagram}

\usepackage[all,cmtip]{xy}

\usepackage[usenames]{color}

\newtheorem{theorem}{Theorem}[section]
\newtheorem{lemma}[theorem]{Lemma}
\newtheorem{corollary}[theorem]{Corollary}

\newtheorem{remark}[theorem]{Remark}
\newtheorem{proposition}[theorem]{Proposition}

\newtheorem{example}[theorem]{Example}

\newtheorem{problem}[theorem]{Problem}

\newproof{proof}{Proof}

\numberwithin{equation}{section}
\numberwithin{theorem}{section}

\newcommand{\e}{\varepsilon}
\newcommand{\w}{\omega}

\newcommand{\NN}{\mathbb{N}}

\newcommand{\IR}{\mathbb{R}}

\newcommand{\DD}{\mathcal{D}}

\newcommand{\KK}{\mathcal{K}}
\newcommand{\Nn}{\mathcal{N}}
\newcommand{\AAA}{\mathcal{A}}

\newcommand{\supp}{\mathrm{supp}}

\newcommand{\Ra}{\Rightarrow}

\newcommand{\LRa}{\Leftrightarrow}

\newcommand{\Id}{\mathsf{id}}

\newcommand{\CC}{C_k}

\newcommand{\SM}{{\setminus}}

\begin{document}

\begin{frontmatter}

\title{Functions on products $X \times Y$ with applications \\ to  Ascoli spaces, $k_\IR$-spaces and $s_\IR$-spaces}

\author{Saak Gabriyelyan}
\ead{saak@math.bgu.ac.il}
\address{Department of Mathematics, Ben-Gurion University of the Negev, Beer-Sheva, Israel}

\author{Evgenii Reznichenko}
\ead{erezn@inbox.ru}
\address{Department of Mathematics, Lomonosov Mosow State University, Moscow, Russia}

\begin{abstract}
We prove that a Tychonoff space $X$ is (sequentially) Ascoli iff  for every compact space $K$ (resp., for a convergent sequence $\mathbf{s}$), each separately continuous $k$-continuous function $\Phi:X\times K\to \IR$ is continuous. We apply these characterizations to show that an open subspace of a (sequentially) Ascoli space is (sequentially) Ascoli, and that the $\mu$-completion and the Dieudonn\'{e} completion of a (sequentially)  Ascoli space are (sequentially) Ascoli. We give also cover-type characterizations of Ascoli spaces and suggest an easy method of construction of pseudocompact Ascoli spaces which are not $k_\IR$-spaces and show that each space $X$ can be closely embedded into such a space. Using a different method we prove Hu\v{s}ek's theorem: a Tychonoff space $Y$ is a locally pseudocompact $k_\IR$-space iff $X\times Y$ is a $k_\IR$-space for each $k_\IR$-space $X$. It is proved that $X$ is an $s_\IR$-space iff  for every locally compact sequential space $K$,  each $s$-continuous function $f:X\times K\to\IR$ is continuous.
\end{abstract}


\end{frontmatter}


\section{Introduction}


All spaces are assumed to be Tychonoff. We denote by $C_p(X)$ and $\CC(X)$ the space $C(X)$ of all continuous real-valued functions on a space $X$ endowed with the pointwise topology or the compact-open topology, respectively. Recall that a space $X$ is called a {\em $k$-space} if for each non-closed subset $A\subseteq X$ there is a compact subset $K\subseteq X$ such that $A\cap K$ is not closed in $K$. The space $X$ is a {\em $k_\IR$-space} if every $k$-continuous function $f:X\to\IR$ is continuous (recall that $f$ is {\em $k$-continuous} if each restriction of $f$ to any compact set $K\subseteq X$ is continuous). Each $k$-space is a $k_\IR$-space, but the converse is false in general. $k$-spaces and $k_\IR$-spaces are widely studied in General Topology (see, for example, the classical book of Engelking \cite{Eng} or  Michael's article \cite{Mi73}). The meaning of $k_\IR$-spaces is explained by the following characterization (see for example Theorem 5.8.7 of \cite{NaB}).

\begin{theorem} \label{t:kR-classical}
A Tychonoff space $X$ is a $k_\IR$-space if, and only if, the space $\CC(X)$ is complete.
\end{theorem}
\smallskip

One of the basic theorems in Analysis is the Ascoli theorem which states that if $X$ is a $k$-space, then every compact subset of $\CC(X)$ is evenly continuous, see Theorem 3.4.20 in \cite{Eng}. In \cite{Noble}, Noble proved that every $k_\IR$-space satisfies the conclusion of the Ascoli theorem.
So it is natural to consider the class of Tychonoff spaces which satisfy the conclusion of Ascoli's theorem. Following \cite{BG}, a  Tychonoff space $X$ is called an {\em Ascoli space} if every compact subset $\KK$ of $\CC(X)$  is evenly continuous, that is the map $X\times\KK \ni(x,f)\mapsto f(x)$ is continuous. In other words, $X$ is Ascoli if, and only if, the compact-open topology of $\CC(X)$ is Ascoli in the sense of \cite[p.45]{mcoy}. One can easily show that $X$ is Ascoli if, and only if, every compact subset of $\CC(X)$ is equicontinuous. Recall that a subset $H$ of $C(X)$ is {\em equicontinuous} if for every $x\in X$ and each $\e>0$ there is an open neighborhood $U$ of $x$ such that $|f(x')-f(x)|<\e$ for all $x'\in U$ and $f\in H$. A more general class of spaces was introduced in \cite{Gabr:weak-bar-L(X)}: a Tychonoff space $X$ is said to be {\em sequentially Ascoli} if every convergent sequence in $\CC(X)$ is equicontinuous. Clearly, every Ascoli space is sequentially Ascoli, but the converse is not true in general (every non-discrete $P$-space is sequentially Ascoli but not Ascoli, see \cite{Gabr:weak-bar-L(X)}). For numerous results concerning the (sequentially) Ascoli property, see \cite{Gabr-B1,Gab-LF,Gabr:weak-bar-L(X),Gabr-seq-Ascoli}.
\smallskip

Analogously to $k_\IR$-spaces, one can define $s_\IR$-spaces as follows. A space $X$ is called  an {\em $s_\IR$-space} if every sequentially continuous function $f:X\to\IR$ is continuous. Recall that a function $f:X\to Y$ between topological spaces $X$ and $Y$ is called {\em sequentially continuous} or {\em $s$-continuous} if the restriction of $f$ onto each convergent sequence in $X$ is continuous. $s$-continuous functions and $s$-continuous functionals appear naturally not only in General Topology, but also in Functional Analysis, see for example Mazur \cite{Mazur1952} and Wilansky \cite{Wilansky}.
The sequential spaces form the most important subclass of the class of  $s_\IR$-spaces. A space $X$ is {\em sequential} if for each non-closed subset $A\subseteq X$ there is a sequence $\{a_n\}_{n\in\NN}\subseteq A$ converging to some point $a\in \bar A\setminus A$. The relationships between the introduced classes of spaces are described in the following diagram in which none of the implications is invertible (see Remarks \ref{rem:kR-perman} and \ref{rem:sR-perman} below):

\[
\xymatrix{
\mbox{sequential} \ar@{=>}[r]\ar@{=>}[rd] &  \mbox{$k$-space} \ar@{=>}[r] &  \mbox{$k_\IR$-space} \ar@{=>}[r]& \mbox{Ascoli}  \ar@{=>}[r] & \mbox{sequentially Ascoli}\\
&  \mbox{$s_\IR$-space} \ar@{=>}[ru] && }.
\]
A characterization of $k_\IR$-spaces is given in Theorem 3.3.5 of Banakh \cite{Banakh-Survey}. Numerous new characterizations of $k_\IR$-spaces and $s_\IR$-spaces are obtained by Gabriyelyan and Reznichenko \cite{GR-kR-sR}.

\smallskip

The aforementioned results and discussion motivate us to study Ascoli spaces, $k_\IR$-spaces and $s_\IR$-spaces, and relationships between these classes of spaces, in particular, for function spaces.

Now we describe the content of the article.
In Section \ref{sec:preliminary}, to make the article self-contained, we recall other notions and some of necessary results used in the article.
In Theorem \ref{t:Acoli-scf} and Theorem \ref{t:seq-Acoli-scf} of Section \ref{sec:Ascoli-spaces} we prove that $X$ is a  (sequentially) Ascoli space if, and only if,  for every compact space $K$ (resp., for a convergent sequence $\mathbf{s}$), each separately continuous $k$-continuous function $\Phi:X\times K\to \IR$ is continuous. We apply these characterizations to show that an open subspace of a (sequentially) Ascoli space is Ascoli (Proposition \ref{p:open-Ascoli}), and that the $\mu$-completion and the Dieudonn\'{e} completion of a (sequentially)  Ascoli space are (sequentially)  Ascoli (Theorem \ref{t:completion-Ascoli}). In Theorem \ref{t:characterization-Ascoli} we give cover-type characterizations of Ascoli spaces. In Proposition \ref{p:Ascoli-non-kR-pseudocompact} and Theorems \ref{t:Tych-embed-ps-Ascoli} and \ref{t:Tych-embed-ps-Ascoli-2} we suggest an easy method of construction of pseudocompact Ascoli spaces which are not $k_\IR$-spaces and show that each space $X$ can be closely embedded into such a space.

In Section \ref{sec:products-kR-sR} we show that analogous characterizations $k_\IR$-spaces and $s_\IR$-spaces but using only $k$-continuous or $s$-continuous functions are valid, see Theorem \ref{t:kR-prod-compact} and Theorem \ref{t:sR-prod-compact}. These characterizations of (sequentially) Ascoli spaces, $k_\IR$-spaces and $s_\IR$-spaces additionally motivate us to study functions on products $X \times Y$ in more details.

Essentially using a result of Whitehead \cite[Lemma~4]{Whitehead}, Cohen proved that {\em if $Z$ is a locally compact Hausdorff space, then $Z\times X$ is a $k$-space for every $k$-space $X$} (see also \cite[3.3.17]{Eng}). In Theorem 4.2 of \cite{Gabr-seq-Ascoli} it is shown that an analogous result is valid for Ascoli spaces: {\em the product of a locally compact space and an Ascoli space is an Ascoli space}. In Section \ref{sec:products-kR-sR} we show that analogous results hold true also for $k_\IR$-spaces (Theorem \ref{p:lc-kR=kR}) and for $s_\IR$-spaces (Theorem \ref{p:sR-product}). Using these theorems we characterize locally pseudocompact $k_\IR$-spaces (Theorem \ref{t:kR-product-mu}) and  locally pseudocompact $s_\IR$-spaces (Theorem \ref{t:lc-sR=sR}).

In the last Section \ref{sec:k-R-spaces} we show in Theorem \ref{t:lcs-Z-kR} that in the definition of $k_\IR$-spaces one can replace $k$-continuous {\em functions} on $X$  by $k$-continuous {\em linear maps} defined on some compact sets of measures.


\section{Preliminary results} \label{sec:preliminary}


Let $X$ be a Tychonoff space. A subset $A$ of $X$ is called {\em functionally bounded} if $f(A)$ is a bounded subset of $\IR$ for every $f\in C(X)$. We shall say that $X$ is {\em closely embedded into } a space $Y$ if there is an embedding $f:X\to Y$ whose image $f(X)$ is a closed subspace of $Y$. A base of the compact-open topology of $\CC(X)$ (resp., a base of the pointwise topology of $C_p(X)$) consists of the sets
\[
[K;\e]=\big\{f\in C(X):f(K)\subseteq (-\e,\e)\big\},
\]
where $K$ is a compact (resp., finite) subset of $X$ and $\e>0$. We denote by $\KK(X)$, $\mathcal{S}(X)$ and $\mathcal{C}(X)$ the family of all compact subsets of $X$, the family of all convergent sequences in $X$ with their limit points, and the family of all countable subsets of $X$.
We denote by $\mathbf{s}=\{\tfrac{1}{n}\}_{n\in\NN}\cup\{0\}$ the convergent sequence with the topology induced from $\IR$.

A space $X$ is called a {\em $\mu$-space} if every functionally bounded subset of $X$ has compact closure. We denote by $\beta X$, $\upsilon X$ and  $\mathcal{D} X$ the Stone-\v{C}ech compactification, the Hewitt realcompactification and the Dieudonn\'{e} completion of $X$, respectively. We denote by $\mu X$ the {\em $\mu$-completion} of $X$, i.e. $\mu X$  is the smallest $\mu$-subspace of $\beta X$ that contains $X$.
It is well-known that
\[
X\subseteq \mu X\subseteq \DD X \subseteq \upsilon X.
\]

Let  $f:X\to Y$ be a continuous mapping between spaces $X$ and $Y$. The {\em adjoint map} $f^{\#}:C_p(Y)\to C_p(X)$ of $f$ is defined by $f^{\#}(g)(x):=g\big(f(x)\big)$.
Following Karnik and Willard \cite{Kar-Wil}, a continuous mapping $p:X\to Y$ is called {\em $R$-quotient} ({\em real-quotient}) if $p$ is continuous, onto, and every function $\phi:Y\to\IR$ is continuous whenever the composition $\phi\circ p$ is continuous. Clearly, every quotient mapping as well as each surjective open mapping is $R$-quotient. The following characterization of $R$-quotient mappings is Proposition 0.4.10 of \cite{Arhangel}.
\begin{proposition}[\cite{Arhangel}] \label{p:R-quotient-characterization}
A continuous surjective mapping $p:X\to Y$ is $R$-quotient if, and only if, $p^{\#}\big(C_p(Y)\big)$ is a closed subspace of $C_p(X)$.
\end{proposition}

We shall use repeatedly the following assertion, see Theorem 4.7 of \cite{Oku-Par}, in which $\Id_Z:Z\to Z$ is the identity mapping.
\begin{theorem}[\cite{Oku-Par}] \label{t:rq-prod}
Let $p:X\to Y$ be an $R$-quotient mapping between Tychonoff spaces. Then for every locally compact space $Z$, the product mapping $p\times\Id_Z:X\times Z\to Y\times Z$ is $R$-quotient.
\end{theorem}

$R$-quotient mappings are important for us because of the following results which are considerably used in what follows.

\begin{theorem}[{\protect\cite[Theorem~6.1]{GR-kR-sR}}] \label{t:characterization-kR}
For a Tychonoff space $X$, the following assertions are equivalent:
\begin{enumerate}
\item[{\rm(i)}] $X$ is a $k_\IR$-space;
\item[{\rm(ii)}]  $X$ is an $R$-quotient image of some topological sum of compact spaces;
\item[{\rm(iii)}]  $X$ is an $R$-quotient  image of a $k_\IR$-space.
\end{enumerate}
\end{theorem}

\begin{theorem}[{\protect\cite[Theorem~7.1]{GR-kR-sR}}] \label{t:characterization-sR}
For a Tychonoff space $X$, the following assertions are equivalent:
\begin{enumerate}
\item[{\rm(i)}] $X$ is an $s_\IR$-space;
\item[{\rm(ii)}]  $X$ is an $R$-quotient image  of a metrizable locally compact space;
\item[{\rm(iii)}]  $X$ is an $R$-quotient  image of an $s_\IR$-space.
\end{enumerate}
\end{theorem}

\begin{proposition}[{\protect\cite[Proposition~4.4]{GR-kR-sR}}]  \label{p:noble:fc-prod}
Let $X=\prod_{\alpha\in \AAA}X_\alpha$ be a product of first countable spaces. If  $|\AAA|$ is not sequential, then  $X$ is an $s_\IR$-space.
Consequently, for every $n\in\omega$, the spaces $\IR^{\aleph_n}$ and $\omega_1^{\aleph_n}$ are $s_\IR$-spaces.
\end{proposition}





\section{Characterizations of Ascoli spaces} \label{sec:Ascoli-spaces}


Two characterizations of Ascoli spaces are given in Propositions 5.4 and 5.10 of \cite{BG}, sequentially Ascoli spaces are characterized in Theorem 2.7 of \cite{Gabr-seq-Ascoli}, and a Banach space characterization of (sequentially) Ascoli spaces is given in Theorem 2.5 of \cite{Gabr-Ascoli-Banach}. Below we give several new characterizations of (sequentially) Ascoli spaces.


\begin{theorem} \label{t:Acoli-scf}
A Tychonoff space $X$ is Ascoli if, and only if, for every compact space $K$, each separately continuous $k$-continuous function $\Phi:X\times K\to \IR$ is continuous.
\end{theorem}

\begin{proof}
Assume that $X$ is an Ascoli space, and let $\Phi:X\times K\to \IR$ be a separately continuous $k$-continuous function for some compact space $K$. Consider the following mapping $\tilde\Phi: K\to \CC(X)$ defined by $\tilde\Phi(f)(x):=\Phi(x,f)$. We claim that $\tilde\Phi$ is continuous. To this end, fix $f\in K$, and let $W:=\tilde\Phi(f)+[A;\e]$ be a standard neighborhood of $\tilde\Phi(f)$, where  $A$ is a compact subset of $X$ and $\e>0$. Since $\Phi$ is $k$-continuous, the restriction of $\Phi$ onto $A\times K$ is continuous. Therefore, for every $a\in A$, there are a neighborhood $U_a\subseteq A$ of $a$ and a neighborhood $V_a\subseteq K$ of $f$ such that $|\Phi(x,g)-\Phi(a,f)|<\tfrac{\e}{2}$ for all $x\in U_a$ and $g\in V_a$. Since $A$ is compact, there are $a_1,\dots,a_n\in A$ such that $A=\bigcup_{i=1}^n U_{a_i}$. Set $V:=\bigcap_{i=1}^n V_{a_i}$, then $V\subseteq K$ is a neighborhood of $f$. Now, fix an arbitrary $g\in V$. For every $x\in A$, choose $1\leq i\leq n$ such that $x\in U_{a_i}$, and then
\[
|\tilde\Phi(g)(x)-\tilde\Phi(f)(x)|\leq |\Phi(x,g)-\Phi(a_i,f)|+|\Phi(a_i,f)-\Phi(x,f)|<\e,
\]
which means that  $\tilde\Phi(g)\in W$ for every $g\in V$. Thus $\tilde \Phi$ is continuous at $f$.

Fix arbitrarily $x_0\in X$ and $f_0\in K$. To show that $\Phi$  is continuous at $(x_0,f_0)$, let $\e>0$. Set $Q:=\tilde\Phi(K)$. Then, by the claim, $Q$ is a compact subset of $\CC(X)$. As $\CC(X)$ is Ascoli, $Q$ is equicontinuous.  Choose a neighborhood $U\subseteq X$ of $x_0$ such that
\[
|\Phi(x,f)-\Phi(x_0,f)|=|\tilde\Phi(f)(x)-\tilde\Phi(f)(x_0)|<\tfrac{\e}{2} \quad \mbox{ for every } \; x\in U\; \mbox{ and each } \; f\in K.
\]
Since $\Phi$ is separately continuous, there is a neighborhood $O\subseteq K$ of $f_0$ such that $|\Phi(x_0,f_0)-\Phi(x_0,f)|<\tfrac{\e}{2}$ for every $f\in O$. Therefore, for each $(x,f)\in U\times O$, we obtain
\[
|\Phi(x,f)-\Phi(x_0,f_0)|\leq |\Phi(x,f)-\Phi(x_0,f)|+|\Phi(x_0,f)-\Phi(x_0,f_0)|<\e.
\]
Thus $\Phi$ is continuous at $(x_0,f_0)$.
\smallskip

Conversely, assume that for every compact space $K$, each separately continuous $k$-continuous function $\Phi:X\times K\to \IR$ is continuous. To prove that $X$ is Ascoli, fix a compact subset $K$ of $\CC(X)$. To show that $K$ is equicontinuous, define $\Phi:X\times K\to \IR$ by $\Phi(x,f):=f(x)$. It is easy to see that the function $\Phi$ is separately continuous. We claim that $\Phi$ is $k$-continuous. To this end, let $A$ be compact subset of $X$. Denote by $R:\CC(X)\to \CC(A)$ the restriction mapping. Since $R$ is continuous, $R(K)$ is a compact subset of $\CC(A)$. As $A$ is Ascoli, the mapping $A\times R(K)\to \IR$ defined by $(x,R(f)):=f(x)=\Phi(x,f)$ is evenly continuous. Hence the mapping $\Phi:A\times K\to\IR$ is continuous. This proves the claim. Thus, by assumption, the function $\Phi$ is continuous.

We show that $K$ is equicontinuous at an arbitrary point $x_0\in X$. Let $\e>0$. For every $f\in K$, choose a neighborhood $V_f\subseteq K$ of $f$ and a neighborhood $U_f\subseteq X$ of $x_0$ such that
\[
|\Phi(x,g)-\Phi(x_0,f)|<\tfrac{\e}{2} \quad \mbox{ for all } x\in U_f \; \mbox{ and }\; g\in V_f.
\]
Since $K$ is compact, there are $f_1,\dots,f_m\in K$ such that $K=\bigcup_{i=1}^m V_{f_i}$. Set $U:= \bigcap_{i=1}^m U_{f_i}$. Then $U$ is a neighborhood of $x_0$. For every $g\in K$, choose $1\leq i\leq m$ such that $g\in V_{f_i}$, and then, for every $x\in U$,  we obtain
\[
\begin{aligned}
|g(x)-g(x_0)| & \leq |g(x)-f_i(x_0)|+|f_i(x_0)-g(x_0)|\\
& =|\Phi(x,g)-\Phi(x_0,f_i)| + |\Phi(x_0,f_i)-\Phi(x_0,g)|<\e.
\end{aligned}
\]
Thus $K$ is equicontinuous at $x_0$.\qed
\end{proof}

Taking into account that a continuous image of a convergent sequence is a convergent sequence, the following characterization of sequentially Ascoli spaces can be proved by a minor modification of the proof of Theorem \ref{t:Acoli-scf} (in fact, only in the proof of the sufficiency we should identify the chosen sequence $K$ with $\mathbf{s}$).

\begin{theorem} \label{t:seq-Acoli-scf}
A Tychonoff space $X$ is sequentially Ascoli if, and only if, each separately continuous $k$-continuous function $\Phi:X\times \mathbf{s}\to \IR$ is continuous.
\end{theorem}

Proposition 2.16 of \cite{Gabr-seq-Ascoli} states that if $p:X\to Y$ is an open surjective map and $X$ is (sequentially) Ascoli, then so is $Y$. Below we generalize this assertion.

\begin{proposition} \label{p:Ascoli-R-quotient}
Let $p:X\to Y$ be an $R$-quotient mapping  between  spaces $X$ and $Y$. If $X$ is a (sequentially) Ascoli space, then so is $Y$.
\end{proposition}

\begin{proof}
To prove that $Y$ is (sequentially) Ascoli, let $K$ be a compact space (resp., $K=\mathbf{s}$), and let $\Phi:X\times K\to\IR$ be a separately continuous $k$-continuous function. By Theorem \ref{t:Acoli-scf} (or Theorem \ref{t:seq-Acoli-scf}), we have to show that $\Phi$ is continuous. It follows from Theorem \ref{t:rq-prod} that the product mapping $p\times\Id_K:X\times K\to Y\times K$ is $R$-quotient. It is clear that the function $\Phi\circ (p\times\Id_K):X\times K\to\IR$ is also separately continuous and $k$-continuous. Since $X$ is a (sequentially) Ascoli space, it follows from Theorem \ref{t:Acoli-scf} (or Theorem \ref{t:seq-Acoli-scf}) that the function $\Phi\circ (p\times\Id_K)$ is continuous. Since the mapping $p\times\Id_K$ is $R$-quotient, we obtain that $\Phi$ is continuous. \qed
\end{proof}

In Theorem of \cite{Blasco-78}, Blasco proved that an open subspace of a $k_\IR$-space is a $k_\IR$-space, too. Analogously, we proved in Proposition 7.6 of \cite{GR-kR-sR} that  an open subspace of an $s_\IR$-space is an $s_\IR$-space. As an application of Theorem \ref{t:Acoli-scf} and Theorem \ref{t:seq-Acoli-scf}, we show that a similar result is valid also for (sequentially) Ascoli spaces.
\begin{proposition} \label{p:open-Ascoli}
An open subspace $U$ of a (sequentially) Ascoli space $X$ is (sequentially) Ascoli.
\end{proposition}

\begin{proof}
Suppose for a contradiction that $U$ is not  (sequentially) Ascoli. Therefore, by Theorem \ref{t:Acoli-scf} (resp., Theorem \ref{t:seq-Acoli-scf}), there are a compact space $K$ (or $K=\mathbf{s}$) and a separately continuous $k$-continuous function $f:U\times K\to \IR$ which is  discontinuous at some point $(x_0,t_0)\in U\times K$. Choose a functionally closed neighborhood $V$ of $x_0$ and a functionally open neighborhood $W$ of $x_0$ such that $V\subseteq W\subseteq \overline{W}\subseteq U$. By Theorem 1.5.13 of \cite{Eng}, choose a continuous function $g:X\to[0,1]$ such that $g^{-1}(1)=V$ and $g^{-1}(0)=X\SM W$. Define a function $h:X\times K\to \IR$ by
\[
h(x,t):=\left\{
\begin{aligned}
f(x,t)\cdot g(x), & \mbox{ if } \; (x,t)\in W\times K,\\
0, &  \mbox{ if } \; (x,t)\in (X\SM W)\times K.
\end{aligned}
\right.
\]
Clearly, $h$ is separately continuous. To show that $h$ is $k$-continuous, let $C$ be a compact subset of $X$, and let $(x,t)\in C\times K$. If $(x,t)\not\in \overline{W}\times K$, then $h$ is continuous at $(x,t)$ because $h=0$ on some neighborhood of $(x,t)$. Assume that $(x,t)\in \overline{W}\times K$. Since $U$ is open and $\overline{W}\subseteq U$, $h$ is continuous at $(x,t)$  as the product of continuous functions $f{\restriction}_{C\times K}$ and $g{\restriction}_C$. Thus $h$ is $k$-continuous.

On the other hand, $h$ is discontinuous at $(x_0,t_0)$ since $h(x,t)=f(x,t)$ on $V\times K$. Therefore, by Theorem \ref{t:Acoli-scf} (resp., Theorem \ref{t:seq-Acoli-scf}), the space $X$ is not (sequentially) Ascoli. This contradiction finishes the proof.\qed
\end{proof}

In the following theorem we provide cover-type characterizations of Ascoli spaces.

\begin{theorem} \label{t:characterization-Ascoli}
For a space $X$, the following assertions are equivalent:
\begin{enumerate}
\item[{\rm(i)}] $X$ is an Acoli space;
\item[{\rm(ii)}] for every {\rm(}the same open{\rm)} subset $A$ of $X$ and each point $x\in \overline{A}$ there is an Ascoli subspace $B$ of $X$ such that $x\in B\cap \overline{A\cap B}$;
\item[{\rm(iii)}] there is a family $\mathcal{M}$ of Ascoli subspaces of  $X$ such that $X$ is strongly functionally generated by $\mathcal{M}$.
\end{enumerate}
\end{theorem}

\begin{proof}
(i)$\Ra$(ii) is clear (just set $B:=X$).

(ii)$\Ra$(i) Suppose for a contradiction that $X$ is not an Ascoli space. Then there exists a compact set $K$ in $\CC(X)$ which is not equicontinuous at some point $z\in X$. Therefore there is $\e_0 >0$ such that for every open neighborhood $U$ of $z$, there exists a function $f_U\in K$ for which the open set $W_{f_U} :=\{ x\in U: | f_U (x)- f_U(z)|> \e_0\}$ is not empty (note that $z\not\in W_{f_U} \subseteq U$). Set
\[
A:= \bigcup\{ W_{f_U}: U \mbox{ is an open neighborhood of } z\}.
\]
Then $A$ is an {\em open} subset of $X$ such that $z\in\overline{A}\setminus A$.
By (ii), choose an Ascoli subspace $B$ of $X$ such that $z\in B\cap \overline{A\cap B}$.
For every $b\in A\cap B$, choose an open neighborhood $U_b$ of $z$ such that $b\in W_{f_{U_b}}$ and, therefore,
\begin{equation} \label{equ:characterization-Ascoli-1}
| f_{U_b} (b)- f_{U_b}(z)|> \e_0 \quad (\mbox{for all } b\in A\cap B).
\end{equation}
Denote by $p$ the restriction map $p: \CC(X)\to \CC(B),\, p(f)=f{\restriction}_{B}$. Then $p(K)$ is a compact subset of the space $\CC(B)$. Since $B$ is Ascoli, $p(K)$ is equicontinuous. In particular, the family $\{ p(f_{U_b}): b\in A\cap B\}\subseteq p(K)$ is equicontinuous at $z\in B$ and, therefore, there is an open neighborhood $V\subseteq X$ of $z$ such that
\[
\big| f_{U_b} (x) - f_{U_b}(z) \big| <\tfrac{\e_0}{2} \; \; \mbox{ for all } \, b\in A\cap B \, \mbox{ and each } \, x\in V\cap B.
\]
Since $z\in \overline{A\cap B}$, for every $b\in V\cap A\cap B$, we obtain
$
\big| f_{U_b} (b) - f_{U_b}(z) \big| <\frac{\e_0}{2}.
$
But this contradicts (\ref{equ:characterization-Ascoli-1}). Thus $X$ is an Ascoli space.
\smallskip

(i)$\Ra$(iii) It suffices to set $\mathcal{M}=\{X\}$.
\smallskip

(iii)$\Ra$(i) Consider the natural mapping $I:\bigoplus \mathcal{M}\to X$, defined by $I(x)=x$ for $x\in M\in \mathcal{M}$. We claim that $I$ is $R$-quotient. Indeed, it is clear that $I$ is continuous and surjective. Now, let $f:X\to\IR$ be a function such that $f\circ I$ is continuous. Assuming that $f$ is discontinuous we could find $M\in\mathcal{M}$ such that $f{\restriction}_M$ is discontinuous, too. Since $f{\restriction}_M=f\circ I{\restriction}_M$ we obtain a contradiction. Thus $f$ is continuous, and hence $I$ is $R$-quotient, as desired.

By Proposition 5.2(2) of \cite{BG}, the space  $\bigoplus\mathcal{M}$ is Ascoli. Thus, by Proposition \ref{p:Ascoli-R-quotient}, $X$ is an Ascoli space.\qed
\end{proof}
Note that the implication (ii)$\Ra$(i) in Theorem \ref{t:characterization-Ascoli} generalizes Theorem 2.5 of \cite{Gabr-B1}.
Using Theorem \ref{t:characterization-Ascoli} we give below  a very short proof of Proposition 5.10 of \cite{BG}.
\begin{corollary}[\cite{BG}] \label{c:dense-Ascoli}
A Tychonoff space $X$ is Ascoli if, and only if, each point $x\in X$ is contained in a dense Ascoli subspace $B_x$ of $X$.
\end{corollary}

\begin{proof}
The necessity is trivial (take $B_x=X$ for every $x\in X$). To prove the sufficiency, let $U$ be an open subset of $X$. Take an arbitrary point $x\in \overline{U}$. Since $B_x\cap U$ is dense in $U$, it follows that $x\in B_x\cap \overline{U\cap B_x}$. Thus, by Theorem \ref{t:characterization-Ascoli}, $X$ is an Ascoli space.\qed
\end{proof}

As an application of Corollary \ref{c:dense-Ascoli}, we obtain Lemma 2.7 of \cite{GGKZ}.
\begin{proposition}[\cite{GGKZ}] \label{p:dense-Ascoli}
Let $Y$ be a dense subset of a homogeneous space (in particular, a topological group or a locally convex space) $X$. If $Y$ is an Ascoli space, then $X$ is also an Ascoli space.
\end{proposition}
\begin{proof}
Let $x\in X$. Take a homeomorphism $h$ of $X$ such that $x\in h(Y)$. Since $B_x:=h(Y)$ is a dense Ascoli subspace of $X$, Corollary \ref{c:dense-Ascoli} implies that $X$ is an Ascoli space. \qed
\end{proof}

Blasco proved in \cite{Blasco-77} that if $X$ is a $k_\IR$-space, then so is $\mu X$. For Ascoli spaces we prove below a stronger result which shows that not only $\mu X$, but also its  Dieudonn\'{e} completion $\DD X$   are Ascoli. We need the following lemma.

\begin{lemma} \label{l:ext-prod-comp}
Let $K$ be a compact space, and let $Y$ be a space. Then any continuous function $f: K\times Y\to\IR$ is extended to a continuous function $\hat f: K\times \DD Y\to\IR$. Moreover, if $K$ is metrizable, then $f$ is extended to a continuous function  $\hat f: K\times \upsilon Y\to\IR$.
\end{lemma}

\begin{proof}
Define $\varphi: Y\to C(K)$ by  $\varphi(y)(x):=f(x,y)$, where $C(K)$ is the Banach space of continuous functions on $K$ with the $\sup$-norm. Since $K$ is compact, $\varphi$ is continuous. Therefore, $\varphi$  is extended to a continuous function $\hat\varphi: \DD Y\to C(K)$  (if $K$ is metrizable, then, by Theorem 3.11.16 of \cite{Eng},  $\varphi$  is extended to a continuous function $\hat\varphi: \upsilon Y\to C(K)$).
Set $\hat f(x,y):=\hat\varphi(y)(x)$ for $(x,y)\in K\times \DD Y$ (resp., $(x,y)\in K\times \upsilon Y$). Then $\hat f(x,y)$ is a desired continuous extension of $f(x,y)$.\qed
\end{proof}

\begin{theorem} \label{t:completion-Ascoli}
Let $X$ be a space, and and let $Y$ be such that $X\subseteq Y\subseteq \beta X$. Then:
\begin{enumerate}
\item[{\rm(i)}] if $X$ is Ascoli and $Y\subseteq \DD X$, then  $Y$ is a Ascoli space;
\item[{\rm(ii)}] if $X$ is sequentially Ascoli  and $Y\subseteq \upsilon X$, then $Y$ is a  sequentially Ascoli space.
\end{enumerate}
\end{theorem}

\begin{proof}
Let $K$ be a compact space (resp., $K=\mathbf{s}$), and let $f:K\times Y$ be a $k$-continuous separately continuous function. Then the restriction $f{\restriction}_{K\times X}$ is also $k$-continuous and separately continuous. Since $X$ is (sequentially) Ascoli, Theorem \ref{t:Acoli-scf} (resp., Theorem \ref{t:seq-Acoli-scf}) implies that $f{\restriction}_{K\times X}$ is continuous. By Lemma \ref{l:ext-prod-comp}, $f{\restriction}_{K\times X}$  is extended to a continuous function $\hat f:K\times Y\to\IR$. As $X$ is dense in $Y$ and $f$ is continuous by the second argument, it follows that $f=\hat f$ is continuous. Once more applying Theorem \ref{t:Acoli-scf} (resp., Theorem \ref{t:seq-Acoli-scf}) we obtain that $Y$ is a (sequentially) Ascoli space.\qed
\end{proof}


If $Y$ is a dense pseudocompact subspace of a space $X$ and $Y$ is (sequentially) Ascoli, then, by Proposition 2.2 of \cite{Gabr-pseudo-Ascoli},   also $X$ is a (sequentially) Ascoli space. This result, Proposition \ref{p:dense-Ascoli} and Theorem \ref{t:completion-Ascoli} motivate the following problem.
Let $Y$ be a dense subspace of a space $X$, and assume that $Y$ is (sequentially) Ascoli. {\em Is then also $X$ (sequentially) Ascoli}?
We answer this question in the negative.

\begin{example} {\em
Let $X=\NN$. Then $X$ is discrete and hence Ascoli. Take an arbitrary point $z\in\beta X\SM X$, and set $Y:=X\cup\{z\}$. Then $X$ is dense in $Y$, does not have infinite compact sets, and  is not a $P$-space (recall that a space  is a {\em $P$-space} if every countable intersection of open sets is open).  Thus, by Corollary 2.2 of \cite{Gabr-lc-Ck}, the space $Y$ is not even sequentially Ascoli.} 
\end{example}

A very complicated example of a pseudocompact Asoli space $Y$ which is not a $k_\IR$-space was constructed by Kato \cite{Kato}. A considerably simpler example was given by Blasco \cite{Blasco-78} (where he constructed an even $\w$-bounded space which is not a $k_\IR$-space). In the next proposition we give a simple method of constructing of pseudocompact Asoli spaces which are not $k_\IR$-spaces.

\begin{proposition} \label{p:Ascoli-non-kR-pseudocompact}
Let $\lambda>\w_1$, $X$ be a subspace of $\big[\tfrac{1}{2},1\big]^\lambda$, and let
\[
S:= \big\{ (x_i)\in [0,1]^\lambda: \, \mbox{ the support } \{i\in \lambda: x_i\not=0\} \mbox{ of } (x_i) \mbox{ is countable}\big\}.
\]
Then the subspace $Y:=S\cup X$ of $[0,1]^\lambda$ satisfies the following conditions:
\begin{enumerate}
\item[{\rm(i)}] $Y$ is a pseudocompact Ascoli space;
\item[{\rm(ii)}] if $K$ is a compact subset of $Y$, then $F:=K\cap S$ and $G:=K\cap X$ are compact;
\item[{\rm(iii)}] $Y$ is not a $k_\IR$-space;
\item[{\rm(iv)}] $X$ is a closed subspace of $Y$.
\end{enumerate}
\end{proposition}

\begin{proof}
(i) It is well-known (and easy to check) that $S$ is a dense sequentially compact subspace of $[0,1]^\lambda$. Therefore, the inclusions $S\subseteq Y\subseteq \DD S=[0,1]^\lambda$ and Theorem \ref{t:completion-Ascoli} imply that  $Y$ is a pseudocompact Asoli space (the same conclusion immediately follows also from Proposition 2.2 and Theorem 1.2 of  \cite{Gabr-pseudo-Ascoli}).
\smallskip

(ii) It suffices to show that the sets $F$ and $G$ are closed.
If $y=(y_i)\in \overline{G}$, then for every $i\in\lambda$, we have $y_i\geq \tfrac{1}{2}$. Therefore, $y\not\in S$ and hence $y\in G$. Thus $G$ is closed.

Assume that $x=(x_i)\in \overline{F}$, and suppose for a contradiction that $x\not\in S$ and hence $x\in X\subseteq \big[\tfrac{1}{2},1\big]^\lambda$.  For every $\alpha<\w_1 <\lambda$, we consider  the natural projection $\pi_\alpha:[0,1]^\lambda \to [0,1]^\alpha$. Since $S$ is sequentially compact and $K$ is compact, it follows that $F$ is sequentially compact. As $\alpha$ is countable, we obtain that $\pi_\alpha(F)$ is a sequentially compact subset of the metrizable compact space $[0,1]^\alpha$. Therefore $\pi_\alpha(F)$ is compact, and hence $\pi_\alpha(x)\in \pi_\alpha(F)=\pi_\alpha(\overline{F})$. Choose $q_\alpha\in F$ such that $\pi_\alpha(x)=\pi_\alpha(q_\alpha)$.

Consider the set $Q:=\{q_\alpha: \alpha<\w_1\}\subseteq F$. As $Q\subseteq K$, the set $Q$ has a cluster point $y=(y_i)\in K$. Taking into account that $\pi_\alpha(x)=\pi_\alpha(q_\alpha)$ for every $\alpha<\w_1$, we obtain that $\pi_{\w_1}(x)=\pi_{\w_1}(y)$. Therefore $y$ has at least $\w_1$ non-zero coordinates and hence $y\not\in S$. On the other hand, the support $A$ of $Q$ has cardinality $\w_1$. As $\w_1 <\lambda$, we obtain that $y_i=0$ for every $i\in\lambda\SM A$. Therefore $y\not\in X$. Hence $y\not\in Y$ which is impossible.
\smallskip

(iii) Consider the characteristic function $\mathbf{1}_{X}$ of $X$. Since $S$ is dense in $Y$,  $\mathbf{1}_{X}$ is discontinuous. It follows from (ii) that  $\mathbf{1}_{X}$ is $k$-continuous.  Thus $Y$ is not a $k_\IR$-space.

(iv) is clear. \qed
\end{proof}

Let us recall some of the most important classes of pseudocompact spaces.
A space $X$ is called
\begin{enumerate}
\item[$\bullet$] {\em sequentially compact} if every sequence in $X$ has a convergent subsequence;
\item[$\bullet$] {\em $\w$-bounded} if every sequence in $X$ has compact closure;
\item[$\bullet$] {\em totally countably compact} if every sequence in $X$ has a subsequence with compact closure;
\item[$\bullet$]  {\em near sequentially compact } if for any sequence $(U_n)_{n\in\w}$ of open sets in $X$ there exists a sequence $(x_n)_{n\in\w}\in\prod_{n\in\w}U_n$ containing a convergent subsequence $(x_{n_k})_{k\in\w}$;
\item[$\bullet$] {\em selectively  $\w$-bounded }  if each infinite collection of {\em disjoint} open sets has an infinite subcollection each of which meets some fixed compact set;
\item[$\bullet$] {\em countably compact} if every sequence in $X$ has a cluster point.
\end{enumerate}
By Theorem 3.2 of \cite{Gabr-seq-Ascoli} and Theorem 1.2 of \cite{Gabr-pseudo-Ascoli}, totally countably compact spaces and near sequentially compact spaces are pseudocompact Ascoli spaces.
The class of selectively  $\w$-bounded spaces is exactly the class $\mathfrak{P}^\ast$ introduced by Frol\'{\i}k \cite{Frolik-tcc}.
By Theorem 1.2 of \cite{Gabr-pseudo-Ascoli}, $X$ is selectively  $\w$-bounded if, and only if, it is a pseudocompact (sequentially) Ascoli space. In the next diagram we summarize the relationships between the above-defined notions and note that none of these implications is reversible: 
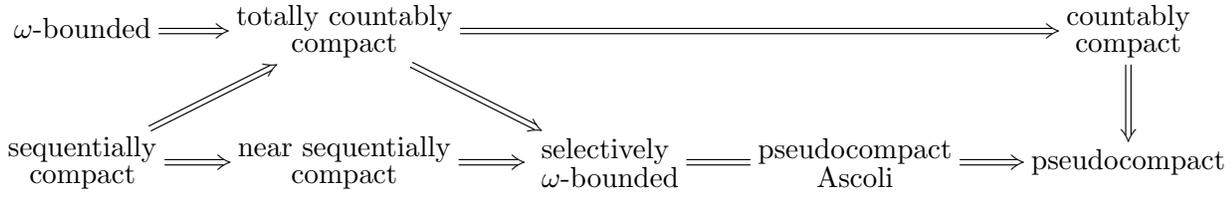
\[
\xymatrix{
\mbox{$\w$-bounded}  \ar@{=>}[r] & {\substack{\mbox{totally countably} \\ \mbox{compact}}}  \ar@{=>}[rd] \ar@{=>}[rrr] &&& {\substack{\mbox{countably} \\ \mbox{compact}}}  \ar@{=>}[d]\\
{\substack{\mbox{sequentially} \\ \mbox{compact}}}  \ar@{=>}[r]  \ar@{=>}[ru] &  {\substack{\mbox{near sequentially} \\ \mbox{compact}}} \ar@{=>}[r]  & {\substack{\mbox{selectively} \\ \mbox{ $\w$-bounded}}}  \ar@{=}[r] &{\substack{\mbox{pseudocompact} \\ \mbox{Ascoli}}}  \ar@{=>}[r] & \mbox{pseudocompact}
}
\]

In Construction 2.3 of \cite{Noble-cc}, Noble proved that any Tychonoff space is closely embedded into a pseudocompact $k_\IR$-space, and that each countably compact space is closely embedded into a countably compact $k_\IR$-space. Below we complement this remarkable Noble's result.

\begin{theorem} \label{t:Tych-embed-ps-Ascoli}
For any infinite Tychonoff space $X$, there is a $C^\ast$-embedding $T$ from $X$ into a pseudocompact Asoli space $Y$ such that
\begin{enumerate}
\item[{\rm(i)}] $Y$ is not a $k_\IR$-space;
\item[{\rm(ii)}] $w(Y)= \max\{w(\beta X), \aleph_2\}$;
\item[{\rm(iii)}] $T(X)$ is a closed subspace of $Y$;
\item[{\rm(iv)}] if, in addition, $X$ is $\w$-bounded (sequentially compact, totally countably compact, near sequentially compact or countably compact), then so is $Y$.
\end{enumerate}
\end{theorem}

\begin{proof}
Let $\beta:X\to \beta X$ be a natural $C^\ast$-embedding, and let $I:\beta X\to \big[\tfrac{1}{2},1\big]^\lambda$ be an embedding, where $\lambda=\max\{w(\beta X), \aleph_2\}>\w_1$. Then $T:=I\circ\beta:X\to  \big[\tfrac{1}{2},1\big]^\lambda$ is a $C^\ast$-embedding. Let $Y:=S\cup T(X)$ with the topology induced from $[0,1]^\lambda$, where  $S$ is the subset of $[0,1]^\lambda$ defined in Proposition \ref{p:Ascoli-non-kR-pseudocompact}. Then (i), (ii) and (iii) immediately follow from Proposition \ref{p:Ascoli-non-kR-pseudocompact}. Observe that  $S$ is  sequentially compact and, clearly, $\w$-bounded. Therefore if $X$ satisfies one of the conditions in (iv), it is easy to see that also $Y$ satisfies the same condition. Thus also (iv) holds true.\qed
\end{proof}

\begin{theorem} \label{t:Tych-embed-ps-Ascoli-2}
Any Tychonoff space $X$ is closely embedded into a pseudocompact Asoli space $Y$ such that
\begin{enumerate}
\item[{\rm(i)}] $Y$ is not a $k_\IR$-space;
\item[{\rm(ii)}] $w(Y)= \max\{w(X), \aleph_2\}$;
\item[{\rm(iii)}] if, in addition, $X$ is $\w$-bounded (sequentially compact, totally countably compact, near sequentially compact or countably compact), then so is $Y$.
\end{enumerate}
\end{theorem}

\begin{proof}
By Theorem 3.5.2 of \cite{Eng}, there is an embedding $T$ of $X$ into a compact space $K$ such that $w(X)=w(K)$. By Theorem 3.2.5 of \cite{Eng}, there is an embedding $I:K\to \big[\tfrac{1}{2},1\big]^{\lambda}$, where $\lambda=\max\{w(X), \aleph_2\}$. Then $I\circ T:X\to  \big[\tfrac{1}{2},1\big]^\lambda$ is an embedding. Let $Y:=\big(I\circ T(X)\big)\cup  S$ with the topology induced from $[0,1]^\lambda$, where  $S$ is the subset of $[0,1]^\lambda$ defined in Proposition \ref{p:Ascoli-non-kR-pseudocompact}. Then (i) and (ii) immediately follow from Proposition \ref{p:Ascoli-non-kR-pseudocompact}. Observe that  $S$ is  sequentially compact and, clearly, $\w$-bounded. Therefore if $X$ satisfies one of the conditions in (iii), it is easy to see that also $Y$ satisfies the same condition. Thus also (iii) holds true.\qed
\end{proof}

\begin{remark} \label{rem:kR-1} {\em
(i) There are $\w$-bounded $s_\IR$-spaces which are not $k$-spaces. Indeed, the product $X=\w_1^{\aleph_2}$ is an $\w$-bounded space (it is easy to see the an arbitrary product of $\w$-bounded spaces is $\w$-bounded). 
The space $X$ is an $s_\IR$-space by Proposition \ref{p:noble:fc-prod}.  The space $X$ is not a $k$-space by Noble \cite{Noble-72} (see also Discussion 2.3 in \cite{Comfort-99}). Note also that, by \cite{Noble-72}, $\w_1^{\aleph_1}$ is a $k$-space.

(ii) There are pseudocompact $s_\IR$-spaces of weight $\w_1$ which are not $k$-spaces. The space $Y:=[0,1]^{\omega_1}\SM \{\mathbf{0}\}$ is a locally compact pseudocompact space of weight $w(Y)=\w_1$. For every $n\in\w$, let $y_n=(a_{in})\in Y$ be such that $a_{in}=1$ if $i=n$, and $a_{in}=0$, otherwise. Then the sequence $S=\{y_n\}_{n\in\w}$ is a closed discrete subset of $Y$. So $Y$ is not countably compact. Set $X:=Y^{\omega_1}$. By Example 7.7 of \cite{GR-kR-sR}, $X$ is an $s_\IR$-space of weight $w(Y)=\w_1$. Since $Y$ contains a closed infinite discrete subset,  $X$ contains a closed subspace topologically isomorphic to  $\w^{\w_1}$ which is not a $k$-space.
%
}
\end{remark}

Theorem \ref{t:Tych-embed-ps-Ascoli-2} and Remark \ref{rem:kR-1} suggest the following problems.
\begin{problem}
Does there exist an $\w$-bounded (totally countably compact, near sequentially compact or countably compact) $k_\IR$-space (or $k_\IR$-group) of weight $\aleph_1$ which is not a $k$-space?
\end{problem}

\begin{problem}
Does there exist an $\w$-bounded (totally countably compact, near sequentially compact or countably compact) Ascoli space (or group) of weight $\aleph_1$ which is not a $k_\IR$-space?
\end{problem}


\section{Functions on products $X \times Y$ with $Y$ being a $k_\IR$-space or an $s_\IR$-space} \label{sec:products-kR-sR}


We start this section with the discussion on products of Ascoli spaces, $k_\IR$-spaces and $s_\IR$-spaces and their relations with $k$-spaces. In particular, it shows that the implications in the diagram in the introduction are not invertible. Recall that a space $X$ is called {\em Fr\'{e}chet--Urysohn} if for any cluster point $a\in X$ of a subset $A\subseteq X$ there is a sequence $\{ a_n\}_{n\in\w}\subseteq A$ which converges to $a$.

\begin{remark} \label{rem:kR-perman} {\em
(i) The product of two first-countable spaces is first-countable. However, if to weaken the first countability to the property of being a Fr\'{e}chet--Urysohn space, then the product can be not a  $k_\IR$-space. In fact,  Proposition 2.14 of \cite{Gabr-seq-Ascoli} states that there is a  Fr\'{e}chet--Urysohn space $X$ whose square $X\times X$ is not even an Ascoli space. On the other hand, by (the proof of) Theorem 1.2 of \cite{Gab-LF}, also the product $\ell_1\times \varphi$ is  not an Ascoli space,  where $\varphi=\IR^{(\w)}$ is the direct locally convex sum of a countable family of $\IR$-s, and, by Nyikos \cite{nyikos},  $\varphi$ is a sequential, non-Fr\'{e}chet--Urysohn space.

(ii) It is well-known that a closed subspace of a $k$-space is also a $k$-space. However, an analogous result for $k_\IR$-spaces is not true in general. Indeed, let $X$ be a countable non-discrete space whose compact sets are finite. Then, by Corollary 2.2 of \cite{Gabr-lc-Ck}, $X$ is not a sequentially Ascoli space. On the other hand, being Lindel\"{o}f the space $X$ is realcompact (\cite[Theorem~3.11.12]{Eng}), and hence it is closely embedded into some power $\IR^\lambda$ (\cite[Theorem~3.11.3]{Eng}) which is a $k_\IR$-space (and a $\kappa$-Fr\'{e}chet--Urysohn space).
If we assume additionally that an Ascoli space (resp., a $k_\IR$-space) $X$ is a stratifiable space, then, by Proposition 5.11 of \cite{BG}, any closed subspace of $X$ also is Ascoli (resp., a $k_\IR$-space).

(iii) Arhangel'skii proved (see \cite[3.12.15]{Eng}) that if $X$ is a hereditary $k$-space (i.e., {\em every} subspace of $X$ is a $k$-space), then $X$ is Fr\'{e}chet--Urysohn. This result was essentially strengthen in Theorem 2.21 of \cite{Gabr-seq-Ascoli} where it is  shown that a hereditary Ascoli space must be Fr\'{e}chet--Urysohn. Consequently, a  hereditary $k_\IR$-space is a Fr\'{e}chet--Urysohn space.

(iv) In \cite{Mi73}, Michael constructed an $\aleph_0$-space which is a $k_\IR$-space but not a $k$-space. Moreover, Borges \cite{Borges-81} proved that this space is also stratifiable.

(v) By Theorem 3.5 of \cite{Frolik-tcc} and Theorem 1.2 of \cite{Gabr-pseudo-Ascoli} (see also Proposition 2.11 of \cite{Gabr-pseudo-Ascoli}), the product $X\times Y$ of a pseudocompact (sequentially) Ascoli  space $X$ and a pseudocompact space $Y$ is pseudocompact.}
\end{remark}

Below we prove an analogue of Cohen's result mentioned in Introduction  for $k_\IR$-spaces.

\begin{theorem} \label{p:lc-kR=kR}
Let $Y$ be a locally compact space, and let $X$ be a $k_\IR$-space. Then $X\times  Y$ is a $k_\IR$-space.
\end{theorem}
\begin{proof}
By Theorem \ref{t:characterization-kR},  there is an $R$-quotient mapping $p:Z\to X$, where $Z$ is  a locally compact space. It is clear that $Z\times Y$ is locally compact, too. By Theorem \ref{t:rq-prod}, the product mapping $p\times \Id_Y\: Z\times Y\to X\times Y$ is $R$-quotient. Applying once more Theorem  \ref{t:characterization-kR} we obtain that $X\times Y$ is a $k_\IR$-space.\qed
\end{proof}

The next theorem is a natural analogue of Theorem \ref{t:Acoli-scf}.
\begin{theorem} \label{t:kR-prod-compact}
For a space $X$, the following assertions are equivalent:
\begin{enumerate}
\item[{\rm(i)}] $X$ is a $k_\IR$-space;
\item[{\rm(ii)}] for every compact space $K$, the product $X\times K$ is a $k_\IR$-space;
\item[{\rm(iii)}] for every compact space $K$,  each $k$-continuous function $f:X\times K\to\IR$ is continuous.
\end{enumerate}
\end{theorem}

\begin{proof}
The implication (i)$\Ra$(ii) follows from Proposition \ref{p:lc-kR=kR}, and the equivalence (ii)$\LRa$(iii) is valid by the definition of $k_\IR$-spaces. Finally, the implication (ii)$\Ra$(i) is trivial (take $K$ to be a one-point space).
\end{proof}

In \cite{Michael-lc}, Michael proved the converse assertion to Cohen's result. So we obtain a characterization of locally compact spaces, see \cite[3.12.14(c)]{Eng}: {\em A Tychonoff space $X$ is locally compact if and only if the product $X\times Y$ is a $k$-space for every $k$-space $Y$.} To obtain a full analogy of this result we use Proposition 4.3 from \cite{Gabr-seq-Ascoli}. Recall that a space $X$ is {\em locally functionally bounded} if every point of $X$ has a functionally bounded neighborhood.

\begin{proposition}[\cite{Gabr-seq-Ascoli}] \label{p:product-Ascoli}
If a space $X$ is not locally functionally bounded, then there is a Fr\'{e}chet--Urysohn fan $Y$ such that the product $X\times Y$ is not an Ascoli space.
\end{proposition}

A space $X$ is {\em locally pseudocompact} if every point of $X$ has a neighborhood which is a pseudocompact space. Clearly, if $X$ is locally pseudocompact, then it is locally functionally bounded. It turns out that both these notions are equivalent.

\begin{lemma} \label{l:lpc=lfb}
A space $X$ is locally pseudocompact if, and only if, it is locally functionally bounded.
\end{lemma}

\begin{proof}
The necessity is trivial. To prove the sufficiency, assume that $X$ is locally functionally bounded. We show that $X$ is locally pseudocompact. Let $x\in X$ and let $U$ be a  functionally bounded neighborhood of $x$. Since subsets and closures of functionally bounded sets are functionally bounded in $X$, we assume that $U$ is open and check that $\overline{U}$ is pseudocompact. Assuming the converse we could find, by Theorem 3.10.22 of \cite{Eng},  a sequence $\{W_n\}_{n\in\w}$ of open in $\overline{U}$ sets which is  discrete in $\overline{U}$. Put $U_n=W_n\cap U$ and choose a point $x_n\in U_n$. As $\{W_n\}_{n\in\w}$ is discrete in the closed subset $\overline{U}$ of $X$, we obtain that the sequence $\{U_n\}_{n\in\w}$  is discrete in $X$. For every $n\in\omega$, choose a continuous function $f_n:X\to [0,n]$ such that $X\setminus f^{-1}_n(0)\subseteq U_n$ and $f_n(x_n)=n$. Then the function $f=\sum_{n\in\omega}f_n$ is well-defined, continuous and unbounded on $\overline{U}$. Thus $\overline{U}$ is not functionally bounded in $X$, a contradiction.\qed
\end{proof}

The equivalence of (i) and (ii) in the next theorem was proved by Hu\v{s}ek \cite[Theorem~1]{Husek}, our proof  completely differs from Hu\v{s}ek's one.
\begin{theorem} \label{t:kR-product-mu}
For a space $Y$ the following assertions are equivalent:
\begin{enumerate}
\item[{\rm (i)}] $Y$ is a locally pseudocompact $k_\IR$-space;
\item[{\rm (ii)}] $X\times Y$ is a $k_\IR$-space for each $k_\IR$-space $X$;
\item[{\rm (iii)}] $X\times Y$ is a $k_\IR$-space space for each Fr\'{e}chet--Urysohn space $X$.
\end{enumerate}
\end{theorem}

\begin{proof}
(i)$\Rightarrow$(ii) Consider the following space
\[
L:=\{x \in \DD Y: x \mbox{  has a compact neighborhood in the space } \DD Y\}.
\]
Then $L$ is a locally compact space, $Y$ is dense and $C$-embedded in $L$.
Let $f: X\times Y\to\IR$ be a $k$-continuous function. Since $X$ and $Y$ are $k_\IR$-spaces, the function $f$ is separately continuous. Let $\hat f: X\times L\to \IR$ be an extension of $f$ such that $\hat f(x,\cdot)$ is continuous for any $x\in X$.

Let us show that $\hat f$ is $k$-continuous. Take a compact subset $K\subseteq X$. We need to prove that $\hat f{\restriction}_{K \times L}$ is continuous. It follows from Proposition \ref{p:lc-kR=kR} that $K\times Y$ is a $k_\IR$-space. Therefore, the function $f{\restriction}_{K\times Y}$ is continuous.
It follows from Lemma \ref{l:ext-prod-comp} that the function $f{\restriction}_{K\times Y}$ is extended to a continuous function $g: K\times \DD Y\to \IR$. Hence, $\hat f{\restriction}_{K \times L}=g{\restriction}_{K \times L}$ is continuous.

It follows from Proposition \ref{p:lc-kR=kR} that $ X\times L$ is a $k_\IR$-space. Since $\hat f$ is $k$-continuous, it follows that $\hat f$ is continuous. Hence, $f$ is continuous.

(ii)$\Rightarrow$(iii) is trivial.

(iii)$\Rightarrow$(i) By Proposition \ref{p:product-Ascoli} and Lemma \ref{l:lpc=lfb}, the space $Y$ is locally pseudocompact. Since the product $X\times Y$ is a $k_\IR$-space space for a one point space $X$, we obtain that  $Y$ is a $k_\IR$-space.\qed
\end{proof}

Now we consider the case of $s_\IR$-spaces.

\begin{remark} \label{rem:sR-perman} {\em
(i) It follows from (i) of Remark \ref{rem:kR-perman} that in general the product of two $s_\IR$-spaces may be not an $s_\IR$-space.

(ii) The compact space $K=\w_1+1$ is not an $s_\IR$-space. Indeed, consider the function $f:X\to \IR$ defined by $f(\lambda):=0$ if $\lambda<\w_1$, and  $f(\w_1)=1$. Clearly, $f$ is $s$-continuous and discontinuous. Thus $K$ is not an $s_\IR$-space. 

(iii) The space $\IR^{\w_1}$ is an $s_\IR$-space (Proposition \ref{p:noble:fc-prod}). The compact space $K=\w_1+1$ has weight $\w_1$ by Theorem 3.1.21 of \cite{Eng}. Therefore, by Theorem 3.2.5 of \cite{Eng},  $K$ embeds into $[0,1]^{\w_1}$  and hence into $\IR^{\w_1}$. By (ii), $K$ is not  an $s_\IR$-space. Therefore  $s_\IR$-spaces may contain compact subspaces which are not  $s_\IR$-spaces.
}
\end{remark}


\begin{theorem} \label{t:lc-sR=sR}
Let $Y$ be a locally  compact $s_\IR$-space, and let $X$ be an $s_\IR$-space. Then $X\times Y$ is an $s_\IR$-space.
\end{theorem}

\begin{proof}
By Theorem \ref{t:characterization-sR}, there is an $R$-quotient mapping $p:Z\to X$, where $Z$ is a metrizable locally compact space.

We claim that the locally compact space $Z\times Y$ is an $s_\IR$-space. Indeed, since $Y$ is an $s_\IR$-space, Theorem \ref{t:characterization-sR} implies that there exists an $R$-quotient mapping $q:Q\to Y$, where $Q$ is a metrizable locally compact space. Therefore, by Theorem \ref{t:rq-prod}, the product mapping $\Id_Z\times q:Z\times Q\to Z\times Y$ is $R$-quotient. Clearly, $Z\times Q$ is also a metrizable locally compact space, and hence it is an $s_\IR$-space.  Now Theorem \ref{t:characterization-sR} implies that  $Z\times Y$ is an $s_\IR$-space, as desired.

By Theorem \ref{t:rq-prod}, the product mapping $p\times \Id_Y:Z\times Y\to X\times Y$ is $R$-quotient. Once more applying Theorem \ref{t:characterization-sR} we obtain that $X\times Y$ is an $s_\IR$-space.\qed
\end{proof}

The next result is analogous to Theorem \ref{t:kR-product-mu}. It is worth mentioning that Theorem 5.2 of \cite{Nob} implies that if $X$ is an $s_\IR$-space and $Y$ is a pseudocompact $s_\IR$-space, then the product $X\times Y$ is an $s_\IR$-space.


\begin{theorem} \label{p:sR-product}
For a space $Y$ the following assertions are equivalent:
\begin{enumerate}
\item[{\rm (i)}] $Y$ is a locally pseudocompact $s_\IR$-space;
\item[{\rm (ii)}] $X\times Y$ is an $s_\IR$-space for each $s_\IR$-space $X$;
\item[{\rm (iii)}] $X\times Y$ is an $s_\IR$-space space for each Fr\'{e}chet--Urysohn space $X$.
\end{enumerate}
\end{theorem}

\begin{proof}
(i)$\Rightarrow$(ii) Set
\[
L:=\{x \in \DD Y: x \mbox{ has a compact neighborhood in the space } \DD Y\}.
\]
Then $L$ is a locally compact space, $Y$ is dense and $C$-embedded in $L$. Let $f: X\times Y\to\IR$ be an $s$-continuous function. Since $X$ and $Y$ are $s_\IR$-spaces, the function $f$ is separately continuous. Let $\hat f: X\times L\to \IR$ be an extension of $f$, such that $\hat f(x,\cdot)$ is continuous for any $x\in X$.

Let us show that $\hat f{\restriction}_{K \times L}$ is continuous for any compact metrizable $K\subseteq X$.
It follows from Theorem \ref{t:lc-sR=sR} that $K\times Y$ is an $s_\IR$-space.
Therefore, the $s$-continuous function $f{\restriction}_{K\times Y}$ is continuous.
It follows from Lemma \ref{l:ext-prod-comp} that the function $f{\restriction}_{K\times Y}$ is extended to a continuous function $g: K\times \DD Y\to \IR$. Hence, $\hat f{\restriction}_{K \times L}=g{\restriction}_{K \times L}$ is continuous.

Let us show that $\hat f$ is continuous. Fix an arbitrary point $(x_0,y_0)\in X\times L$ and check that $\hat f$ is continuous at $(x_0,y_0)$. Let $U$ be a neighborhood of the point $y_0$ such that $S:=\overline{U}$ is compact. It suffices to verify that $\hat f{\restriction}_{X\times S}$ is continuous.
Put
\[
\varphi: X\to C(S),\; \varphi(x)(y)=\hat f(x,y).
\]
Since $\hat f{\restriction}_{K \times S}$ is continuous for any compact metrizable $K\subset X$, then $\varphi$ is $s$-continuous. As $X$ is an $s_\IR$-space, it follows that $\varphi$ is continuous. Therefore, $\hat f{\restriction}_{X\times S}$ is continuous and hence the function $\hat f$ is continuous.

Thus, the function $f=\hat f{\restriction}_{X\times Y}$ is continuous, too.
\smallskip

(ii)$\Rightarrow$(iii) is trivial.
\smallskip

(iii)$\Rightarrow$(i) By Theorem \ref{t:kR-product-mu},  $Y$ is locally pseudocompact $k_\IR$-space. By (iii) applied to a one point space $X$,  it follows that  $Y$ is an $s_\IR$-space.\qed
\end{proof}

The next theorem is an analogue of Theorem \ref{t:Acoli-scf} and Theorem \ref{t:kR-prod-compact}.
\begin{theorem} \label{t:sR-prod-compact}
For a space $X$, the following assertions are equivalent:
\begin{enumerate}
\item[{\rm(i)}] $X$ is an $s_\IR$-space;
\item[{\rm(ii)}] for every locally compact sequential space $K$, the product $X\times K$ is an $s_\IR$-space;
\item[{\rm(iii)}] for every locally compact sequential space $K$,  each $s$-continuous function $f:X\times K\to\IR$ is continuous.
\end{enumerate}
\end{theorem}

\begin{proof}
The implication (i)$\Ra$(ii) follows from Theorem \ref{t:lc-sR=sR}, and the implication (ii)$\Ra$(i) is trivial (take $K$ to be a one-point space). The equivalence (ii)$\LRa$(iii) is valid by the definition of $s_\IR$-spaces.\qed
\end{proof}

Theorem \ref{t:kR-product-mu} and Theorem \ref{p:sR-product}  motivate the following problem.

\begin{problem}
Let $X$ be a (locally) pseudocompact Ascoli space (for example, $X=\omega_1$), and let $Y$ be an Ascoli space. Is then $X\times Y$ Ascoli?
\end{problem}
Note that, by Proposition 2.7 of \cite{Gabr-pseudo-Ascoli}, the product of an arbitrary family of  pseudocompact Ascoli spaces is a pseudocompact Ascoli space.


\section{A characterization of $k_\IR$-spaces} \label{sec:k-R-spaces}


In the next theorem we show that in the definition of $k_\IR$-spaces one can replace $k$-continuous {\em functions} on $X$  by $k$-continuous {\em linear maps} defined on some compact sets of measures.
 To this end, we need some notations.
Recall (see Corollary 7.6.5(a) of \cite{Jar}) that the dual space of $\CC(Z)$ is the space $M_c(Z)$ of all regular Borel measures on $Z$ with compact support.
A  linear form $\chi$ on $M_c(Z)$ is called {\em $\gamma$-continuous} if for every equicontinuous weak$^\ast$  closed disk $D$ in $M_c(Z)$, the restriction $\chi{\restriction}_D$ is weak$^\ast$ continuous.

\begin{theorem} \label{t:lcs-Z-kR}
For a Tychonoff space $Z$, the following assertions are equivalent:
\begin{enumerate}
\item[{\rm(i)}] $Z$ is a $k_\IR$-space;
\item[{\rm(ii)}] each linear form $\chi$ on $M_c(Z)$ which is continuous on each weak$^\ast$ compact subset of the form
\begin{equation} \label{equ:p:lcs-Z-kR-1}
B_{C(K)^\ast}= \{\mu\in M_c(Z): \supp(\mu)\subseteq K \mbox{ and } \|\mu\|\leq 1\}, \;\; \mbox{ where $K\in \KK(Z)$, }
\end{equation}
 is weak$^\ast$ continuous {\rm(}i.e. $\chi\in C(Z)${\rm)}.
\end{enumerate}
\end{theorem}

\begin{proof}
Let a linear form $\chi$ on $M_c(Z)$ be continuous on each weak$^\ast$ compact subset of the form (\ref{equ:p:lcs-Z-kR-1}). Observe that any  equicontinuous weak$^\ast$  closed disk $D$ in $M_c(Z)$ is contained in $U^\circ$ for some neighborhood $U$ of zero in $\CC(Z)$. On the other hand, each set of the form  $U^\circ$ is an equicontinuous weak$^\ast$  closed disk. So we can assume that $\chi$ is $\gamma$-continuous if and only if  $\chi{\restriction}_{U^\circ}$ is weak$^\ast$ continuous for every $U\in \Nn_0\big(\CC(Z)\big)$. Moreover, we can assume additionally that $U$ has the form
\[
U=[K;\e] :=\{ f\in C(Z): K \mbox{ is compact in $Z$ and } f(K)\subseteq (-\e,\e)\big\}.
\]
Therefore
\[
U^\circ=\{ \mu\in M_c(Z): \supp(\mu)\subseteq K \mbox{ and } \|\mu\|\leq \tfrac{1}{\e}\big\}.
\]
Taking into account the homogeneity of locally convex spaces, we can say the following: A  linear form $\chi$ on $M_c(Z)$ is $\gamma$-continuous if, and only if, the restriction $\chi{\restriction}_{B_{C(K)^\ast}}$ is weak$^\ast$ continuous, where $K\in\KK(Z)$ and  $B_{C(K)^\ast}$ denotes the closed unit ball of the dual space of the Banach space $C(K)$ endowed with the weak$^\ast$ topology.

By Theorem \ref{t:kR-classical},  $Z$ is a $k_\IR$-space if and only if the space $\CC(Z)$ is complete. By Grothendieck's Completeness Theorem 9.2.2 of \cite{Jar}, $\CC(Z)$ is complete if and only if each $\gamma$-continuous linear form on $\CC(Z)'=M_c(Z)$ belongs to $\CC(Z)$. Now the theorem follows from the discussion in the previous paragraph.\qed
\end{proof}

\bibliographystyle{amsplain}

\end{document}